\documentclass[11pt]{article}
\usepackage{amsmath,amssymb}
\usepackage{graphicx, enumerate}
\usepackage{chngcntr}
\usepackage{epstopdf}
\usepackage{epsfig}
\usepackage{setspace}
\usepackage{indentfirst}
\usepackage{amsthm}
\usepackage{makecell}
\usepackage{multirow}
\usepackage{geometry}
\counterwithin{table}{section}

\newtheorem{thm}{Theorem}[section]

\newtheorem{lem}[thm]{Lemma}

\newcounter{SectionStep}[section]

\newcounter{SideStep}[SectionStep]

\usepackage{hyperref}
\usepackage{authblk}
\begin{document}
\allowdisplaybreaks
\title{The maximum number of two $4$-vertex graphs in planar graphs}
\author{Wei Liu$^{1}$, Lin Sun$^{2}$\thanks{This work was supported by Guangdong Basic and Applied Basic Research Foundation (Natural Science Foundation of Guangdong Province, China, Grant No. 2025A1515011904).}, Jianliang Wu$^{1}$\thanks{This work is partially supported by National Natural Science Foundation of China (11971270, 11631014).}\thanks{Corresponding author. \emph{E-mail address:} jlwu@sdu.edu.cn.}\\
{\small $^1$\ School of Mathematics, Shandong University, Jinan 250100, China}\\
{\small $^2$\ School of Mathematics and Statistics, Lingnan Normal University, Zhanjiang 524000, China.}\\
}
\date{}

\maketitle

\begin{abstract}
Let $f(n,H)$ be the maximum number of copies of a graph $H$ in a planar graph of order $n$. When $H$ is a connected graph on four vertices, $f(n,H)$ has been completely determined except for two cases: $K_{1,3}^+$ (the claw graph $K_{1,3}$ with one additional edge) and $K_4^-$ (the complete graph $K_4$ with one edge removed). Here, we address these two cases and establish that for all $n\ge4$,
$$f(n,K_{1,3}^+) = 4n^2-12n-4 \textrm{ and } f(n,K_4^-) =\frac{1}{2}(n^2+9n-40).$$ 
We also characterize all planar graphs attaining these bounds. 
\end{abstract}



\section{Introduction}\label{sec:Intro}
\noindent
All graphs considered in this paper are simple and finite. Any undefined terminology follows that of Diestel \cite{GraphTheory}. Let $G=(V,E)$ be a graph. We use $|G|$, $e(G)$, $\delta(G)$ and $\overline{G}$ to denote the number of vertices, number of edges, minimum degree and the complement of $G$, respectively. If $|G|=k$, we say $G$ is a $k$-$vertex$ $graph$.
For a subset $S\subseteq V(G)$, we use $G[S]$ to denote the subgraph of $G$ induced by $S$ and $G - S = G[V(G)\setminus S]$. The \textit{neighborhood} of $S$ is defined as $N(S) =\{u\in V(G) \setminus S : \;$ there exists $v\in S$ such that $uv\in E(G)\}$. For a vertex $x\in V(G)$, we write $G-x := G-\{x\}$, $N_G(x) := N_G(\{x\})$, $N_G[x]:= N_G(x)\cup \{x\}$ and the \textit{degree} $d_G(x):=|N_G(x)|$. If $d_G(x) = k$, $x$ is called a $k$-$vertex$.
A \textit{pendant vertex} of $G$ is a $1$-vertex and a \textit{pendant edge} is an edge incident with a pendant vertex. The subscripts are omitted when unambiguous. 

For two graphs $G$ and $H$, let $\mathcal{N}(H, G)$ denote the number of subgraphs of $G$ isomorphic to $H$ (called \textit{copies} of $H$ or $H$-\textit{copies}), and let $\mathcal{N}_x(H, G)$ denote the number of $H$-copies containing $x$ in $G$ for any vertex $x\in V(G)$. Given a family of graphs $\mathcal{F}$, we say $G$ is $\mathcal{F}$-free if it contains no subgraph isomorphic to any graph in $\mathcal{F}$. 
The \textit{Turán number} of $\mathcal{F}$, $ex(n,\mathcal{F})$, denotes the maximum number of edges of an $n$-vertex $\mathcal{F}$-free graph. This function has been intensively studied, starting with the Turán's Theorem \cite{Turan1941} that determines $ex(n,K_r)$ for $r>1$. Turán also showed \cite{Turan1941} that the unique extremal graph is the complete $n$-vertex $r$-partite graph with as equal parts as possible. This graph is called the \textit{Turán graph} and is denoted by $T_r(n)$. 

Alon and Shikhelman \cite{alon2016many} introduced the following \textit{generalized Turán number}, 
$$
ex(n, H, \mathcal{F}) = \max \{\mathcal{N}(H, G) : G \text{ is an }n \text{-vertex } \mathcal{F}\text{-free graph}\}.
$$
Simply write $ex(n,H, F)= ex(n, H, \{F\})$. They proved that for any graph $F$ with chromatic number $\chi(F)=r+1>m$, $ex(n,K_m,F)=\mathcal{N}(K_m, T_r(n))+o(n^m)$. 
When $H=P_2$ in $ex(n, H, \mathcal{F})$, this reduces to the classical Turán problem. 
For $H\neq P_2$, the earliest result is due to Zykov \cite{zykov1949certain}, who determined $ex(n, K_s, K_t)$ exactly for all $s$ and $t$. Erdős conjectured that $ex(n, C_5, C_3) \sim (\frac{n}{5})^5$, which was later confirmed independently by Hatami et al. \cite{hatami2012number} and by Grzesik \cite{grzesik2012maximum}. Győri and Li \cite{Gyori2012} proved that for $k\ge2$, $ex(n,C_3,C_{2k+1})\le \frac{(2k-2)(16k-8)}{2}ex(n,C_{2k})$. 
Recently, Grzesik and Kielak \cite{grzesik2022odd} determined the asymptotic value of $ex(n, C_k, C_{k-2})$ for every odd $k$. The function $ex(n, C_3, C_5)$ was studied by Bollobás and Győri \cite{bollobas1988maximal}. Their results were subsequently improved in \cite{alon2016many,Ergemlidze2019,Ergemlidze2022}, although the correct asymptotic of the function remains open. The problem of maximizing the number of $P_\ell$ copies in a $P_k$-free graph was investigated in \cite{Pell}. Recently, the maximum number of copies of $K_{s,s}$ in a $C_{2s+2}$-free $n$-vertex graph for all integers $s\ge2$ and sufficiently large $n$ is determined in \cite{Gyori2025}. 
For research regarding the case where \(\mathcal{F}\) is the set of disconnected graphs, refer to \cite{Chen2024,Gerbner2019,Liu2025,Wang2020,Zhang2022,Zhu2022,Zhu2021}.

Let $G$ be a graph and $xy\in E(G)$.  To \emph{contract} $xy$ is to delete it first and then replace $x$ and $y$ by a single vertex incident with all the edges which were incident in $G$ with either $x$ or $y$, finally delete one edge of any pair of parallel edges so formed. A graph $H$ is a \emph{minor} of a graph $G$ (or $G$ contains an $H$-minor) if $G$ has a subgraph contractible to $H$. $G$ is called $H$-\emph{minor free} if $G$ does not have $H$ as a minor. Let $\mathcal{F}$ be the set of graphs each of which contains an $K_5$-minor or $K_{3,3}$-minor, and
$$
f(n, H)= ex(n, H, \mathcal{F}) = \max \{\mathcal{N}(H, G) : G \text{ is an } n \text{-vertex planar graph}\}.
$$
Let
$$\mathbb{F}(n, H) = \{G: G\text{ is an } n \text{-vertex planar graph}\text{ and }\mathcal{N}(H, G) = f(n, H)\}$$
and 
$$\mathbb{F}(H)= \bigcup_{n\ge |H|}\mathbb{F}(n, H).$$

\begin{figure}
 \centering
 \includegraphics[height=5cm]{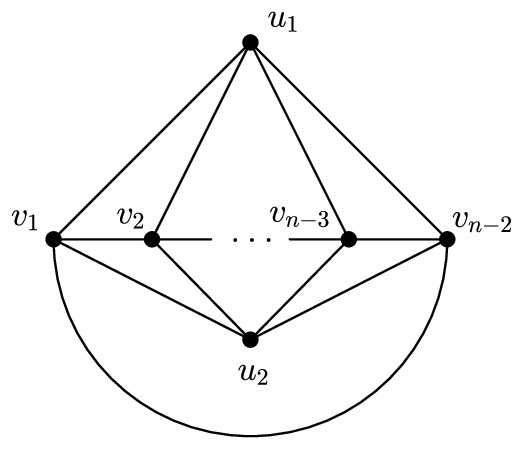}
 \caption{The double-wheel graph $WW_n$.}\label{Fig:WWn}
\end{figure} 

In the following, we always assume that all planar graphs are embedded in the plane. For a face $f$ of a planar graph, let $\partial f$ denote the set of vertices on the boundary of $f$. 
In the process of exploring the extremal graphs for $f(n,H)$, researchers have discovered several special planar graphs. 
A $double$-$wheel$ graph $WW_n$ on $n(\ge 5)$ vertices is the graph obtained from an $(n-2)$-cycle by adding two non-adjacent vertices adjacent to all vertices of the cycle (see Figure \ref{Fig:WWn}).
An \textit{Apollonian graph} is constructed from $K_3$ by recursively placing a vertex inside a face and joining it to the three vertices of that face. Let $H_n$ be the graph obtained from a path $P_{n-2}$ by adding two adjacent vertices, each adjacent to all vertices of the path (see Figure \ref{Fig:Fn}$(a)$). Let $J_1$ and $J_2$ be the graphs shown in Figure \ref{Fig:Fn}$(b)$ and $(c)$, respectively. Denote $\mathcal{A}_n=\{G:\;G$ is an $n$-vertex Apollonian graph$\}$, $\mathcal{A} = \bigcup_{n\ge3}\mathcal{A}_n$, $\mathcal{H} = \{H_n:\;n\ge 3\}$ and $\mathcal{H}^+ = \{J_1,J_2\}\cup\mathcal{H}\setminus\{H_3\}$. 
\begin{figure}
 \centering
 \includegraphics[scale=0.8]{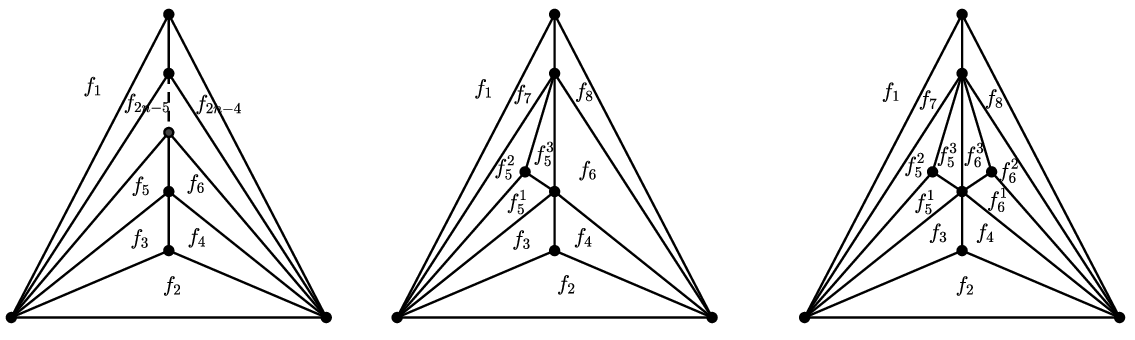}\\
 $(a)\;\; H_n$\qquad\qquad\qquad\qquad\qquad\qquad$(b) \;\; J_1$\quad\qquad\qquad\qquad\qquad\qquad$(c) \;\; J_2$
 \caption{Special planar graphs.}\label{Fig:Fn}
\end{figure}

It is well-known that $f(n, K_2) = 3n -6(n\ge 3)$ and $\mathbb{F}(K_2)$ is the set of all maximal planar graphs, that is, all planar triangulations. Alon and Caro \cite{Alon1984Caro} demonstrated that $f(n, P_3) = 2\binom{n-1}{2}+ 6n-18$ for any $n\ge4$. 
Hakimi and Schmeichel \cite{hakimi1979} showed that $f(n, C_3) = 3n - 8$ for any $n\ge3$ and $\mathbb{F}(C_3)=\mathcal{A}$. 

If $H$ is a 4-vertex connected graph, then $H\in\{K_{1,3}, P_4, C_4, K_{1,3}^+, K_4^-, K_4\}$ where $K_{1,3}^+$ is the claw graph $K_{1,3}$ with an additional edge and $K_4^-$ is the complete graph $K_4$ with one edge removed. Alon and Caro \cite{Alon1984Caro} proved that for any $n\ge4$, $$f(n,K_{1,3})=2\binom{n-1}{3}+4n-14,\text{ and }\mathbb{F}(K_{1,3})=\mathcal{H}\setminus\{H_3\}. $$ 
Alameddine \cite{Alameddine1980}, Hakimi and Schmeichel \cite{hakimi1979} and Alon and Caro \cite{Alon1984Caro}, respectively proved that for any $n\ge4$, 
$$ f(n,C_4) = f(n,K_{2,2}) =\binom{n-2}{2}+4n-14\text{ and } \mathbb{F}(C_4) = \mathcal{H}^+,$$ 
Wood \cite{Wood2007} proved that for any $n\ge4$, $f(n,K_4)= n-3$. Grzesik et al. \cite{grzesik2022P4} proved that for any $n\ge4$, 
$$f(n,P_4)=\begin{cases}
     12, & \text{if }n=4,\\
     147, & \text{if }n=7,\\
     222, & \text{if }n=8,\\
     7n^2-32n+27, & \text{if }n=5,6\text{ and }n\ge9,
   \end{cases}$$
and $\mathbb{F}(P_4)=\mathcal{H}^+\setminus\{H_7,H_8\}$. 


Moreover, Antonir and Shapira \cite{Antonir2024} provided the upper bounds for $f(n,C_{2m+1})$ for all integers $m\ge2$. Heath et al.  \cite{Heath2025} established the asymptotic bounds for $f(n,C_{2m+1})$ when $m\in\{2,3,4\}$, and provided the upper bounds for all other values of $m\ge5$. 

We newly determine the extremal number of $K_{1,3}^+$, while the result for $K_4^-$ was first obtained by the authors in an SSRN preprint \cite{ShiWang2025}. We shorten its original twenty-page proof to under one page. Our main results are as follows.

\begin{thm}\label{thm:K13+}
$f(n,K_{1,3}^+)= 4n^2-12n-4$ for any $n\ge4$, and $\mathbb{F}(K_{1,3}^+)=\mathcal{H}^+$.
\end{thm}

\begin{thm}\label{thm:K4-}
$f(n,K_4^-)=\frac{1}{2}(n^2+9n-40)$ for any $n\ge4$, and $\mathbb{F}(K_4^-)=\mathcal{H}^+$. 
\end{thm}

In Section \ref{sec:proof_K13+}, we give the proof of Theorem \ref{thm:K13+}. In Section 3, we prove Theorem \ref{thm:K4-}. In the last section, we establish full characterizations of the graphs in $\mathbb{F}(P_3)$ and $\mathbb{F}(K_4)$. 
Alon and Caro pointed (see Remark 1 in \cite{Alon1984Caro}) that $\mathbb{F}(P_3)=\mathcal{H}$. However, this assertion is not correct. In fact, $\mathbb{F}(P_3)=\mathcal{H}\cup\{J_1,J_2\}=\mathcal{H}^+\cup\{H_3\}$. 
Wood \cite{Wood2007} proved that the graphs in $\mathcal{A}\setminus\{K_3\}$ contain exactly $f(n,K_4)= n-3$ copies of $K_4$, that is, $\mathcal{A}\setminus\{K_3\}\subseteq \mathbb{F}(K_4)$.  Here, we show that $\mathbb{F}(K_4)=\mathcal{A}\setminus\{K_3\}$. 
Together with earlier results, this completes the characterization of both $f(n,H)$ and $\mathbb{F}(H)$ for all connected graphs $H$ with $|H|\in\{3,4\}$. 

To prove our results, we need the following useful lemmas.

\begin{lem}[{\rm \cite{Campos2013}}]\label{lem:op_delta}
Let $G$ be an outerplanar graph. Then $\delta(G)\le2$. 
\end{lem}

\begin{lem}[{\rm \cite{dirac1964}}]\label{lem:K4minor}
Let $G$ be a $K_4$-minor-free graph of order $n\ge4$. Then $e(G)\le 2|G|-3$.
\end{lem}

\begin{lem}[\cite{Matolcsi2022}]\label{lem:exop_P3}
Let $G$ be an $n$-vertex outerplanar graph with $n\ge3$. Then $$\mathcal{N}(P_3,G)\le\frac{n^2+3n-12}{2}.$$
\end{lem}

\begin{lem}\label{lem:maximal}
Let $H$ be a graph with $H\in \{P_3,K_{1,3}^+,K_4^-\}$, and $G$ be an $n$-vertex planar graph with $n\ge4$ and $\mathcal{N}(H, G) = f(n, H)$, then $G$ is a maximal planar graph, that is, a triangulation.
\end{lem}

\begin{proof}
Suppose that $G$ is not maximal. 
Let $G'$ be a triangulation with $G\subset G'$. 
Let $e=uv\in E(G'-G)$. 
Since $G'$ is a triangulation, $uv$ belongs to two triangular faces, and we may choose a copy $J$ of $K_4^-$ in $G'$  containing $uv$ such that $d_J(u)=d_J(v)=3$. 
Since $H\in \{P_3,K_{1,3}^+,K_4^-\}$, $J$ contains an $H$-copy containing $e$, which gives $\mathcal{N}_e(H,G')>0$.
Thus, $\mathcal{N}(H,G)\le \mathcal{N}(H,G'-e)<\mathcal{N}(H,G'-e)+\mathcal{N}_e(H,G')=\mathcal{N}(H,G')$, a contradiction.
\end{proof}

\begin{lem}\label{lem:maximize}
Let $(x_1,\dots,x_n)$ be an integer sequence with $4\le x_1\le\dots\le x_n\le n-2$ and $\sum_{i=1}^{n}x_i=6n-12$. Then $$\sum_{i=1}^{n}x_i^2\leq 2n^2+8n-24,$$ with equality if and only if $(x_1, x_2, \dots, x_n)=(4,\dots,4,n-2,n-2)$.
\end{lem}
\begin{proof}
Let $(t_1,\dots,t_n)$ be an integer sequence with $4\le t_1\le\dots\le t_n\le n-2$ that maximizes $\sum_{i=1}^{n}t_i^2$ subject to $\sum_{i=1}^{n}t_i=6n-12$. Let $q$ denote the number of elements in $(t_1,\dots,t_n)$ strictly between $4$ and $n-2$.
If $q\geq 2$, then there exist indices $1\le j<k\le n$ such that $4< t_j \le t_k <n-2$. Let $(t_1',\dots,t_n')$ be obtained from $(t_1,\dots,t_n)$ by replacing $t_j$ with $t_j-1$ and $t_k$ with $t_k+1$, and then reordering if necessary so that $t_1' \le \dots \le t_n'$. 
Then $(t_1',\dots,t_n')$ satisfies $4\le t'_i\le n-2$ for all $i$ and $\sum_{i=1}^{n}t'_i=6n-12$.
Moreover,
$\sum_{i=1}^n (t_i')^2= \sum_{i=1}^n t_i^2 +2(t_k - t_j + 1)> \sum_{i=1}^n t_i^2$,
a contradiction.
Hence, $q\leq 1$. 

If $t_{n-1}<n-2$, then $\sum_{i=1}^{n}t_i\le 4(n-2)+ (n-3)+(n-2)<6n-12=\sum_{i=1}^{n}t_i$, a contradiction. So $t_{n-1} =t_n= n-2$.
If $t_{n-2}>4$, then $\sum_{i=1}^{n}t_i\ge 4(n-3)  + 5+ 2(n -2)= 6n-12 + 1>\sum_{i=1}^{n}t_i$, a contradiction. So $t_1=\dots=t_{n-2}=4$. 
Thus, $(t_1,\dots,t_n)=(4,\dots,4,n-2,n-2)$. It follows that
$\sum_{i=1}^{n} x_i^2\le\sum_{i=1}^{n} t_i^2=2(n-2)^2+ 4^2(n-2) = 2n^2+8n-24$,
with equality if and only if $(x_1, x_2, \dots, x_n)=(4,\dots,4,n-2,n-2)$.
\end{proof}

\begin{lem}\label{lem:delta4_d2x}
Let $G$ be an $n$-vertex planar graph with $\delta(G)\ge4$. Then $\Delta(G)\le n-2$ and $$\sum_{x\in V(G)}d^2(x)\le 2n^2+8n-24,$$ with equality if and only if there are two $(n-2)$-vertices and $(n-2)$ $4$-vertices.
\end{lem}
\begin{proof}
If $\Delta(G) = n-1$, that is, there is a vertex $v\in V(G)$ with $d(v)=n-1$, then $G-v$ is an outerplanar graph. By Lemma \ref{lem:op_delta}, $\delta(G-v)\le 2$. This implies that $\delta(G)\le 3$, a contradiction. So $\Delta(G)\le n-2$. 

Without loss of generality, assume that $G$ is a maximal planar graph. Then $2e(G)=\sum_{i=1}^{n}d_i=6n-12$. Let $(d_1, d_2, \dots, d_n)$ be the degree sequence of $G$ such that $4\le d_1\le\dots\le d_n\le n-2$. 
By Lemma \ref{lem:maximize}, $\sum_{i=1}^{n} d_i^2\le 2n^2+8n-24$, with equality if and only if $(d_1, d_2, \dots, d_n)=(4,\dots,4,n-2,n-2)$. At the same time, this degree sequence is realized by the double-wheel graph $WW_n$.
\end{proof}

\section{Proof of Theorem \ref{thm:K13+}}\label{sec:proof_K13+}
\begin{lem}\label{lem:N_K13+}
Let $G$ be an $n$-vertex graph. Then $$\mathcal{N}(K_{1,3}^+,G) =\sum\limits_{xy\in E(G)}(e(G[N(x)-y])+e(G[N(y)-x]).$$
\end{lem}
\begin{proof}
Let $xy\in E(G)$. If there exists a copy $H$ of $K_{1,3}^+$ in $G$ with $V(H)=\{x, y, u,v\}$ and $d_H(x)=1$, then $uv\in E(G)$ and $\{u,v\}\subseteq N_G(y)$. So $e(G[N(y)-x])$ is just the number of $K_{1,3}^+$-copies in which $x$ is the pendant vertex. Hence $\mathcal{N}(K_{1,3}^+,G) =\sum_{xy\in E(G)}(e(G[N(x)-y])+e(G[N(y)-x]).$
\end{proof}

\begin{lem}\label{lem:K13+,delta4}
Let $G$ be an $n$-vertex maximal planar graph and $\delta(G)\ge4$. Then 
$$\mathcal{N}(K_{1,3}^+,G) < 4n^2-12n-4.$$
\end{lem}
\begin{proof}
By Lemma \ref{lem:delta4_d2x}, $\Delta(G)\le n-2$.
Since $G$ is planar and $\delta(G)\ge 4$, $n\ge 6$. Suppose that $n =6$. Then $d(x) =4$ for all $x\in V(G)$, that is, $G$ is a $4$-regular graph and it follows that $G=K_{2,2,2}=WW_6$, which is maximal with respect to having no $K_5$-minors. Thus $e(G[N(x)]-y]) = 2$ for any $xy\in E(G)$. By Lemma \ref{lem:N_K13+}, $\mathcal{N}(K_{1,3}^+,G)=4\times e(G) =48<4n^2-12n-4=68$. 

Suppose that $n = 7$. Then $4\le d(v) \le 5$ for any $v\in V(G)$. Since $e(G) = 3n - 6$, $G$ has two 5-vertices and five $4$-vertices. Let $v$ be a $5$-vertex of $G$ and $\{u\} = V(G)\setminus N[v]$. Since $\delta(G)\ge 4$, $G[N(v)]$ is an outerplanar graph with the minimum degree $2$ and contains a 5-cycle $C = v_1v_2\dots v_5v_1$. If $C$ has a chord, say $v_1v_3$, then $u$ must be in the face formed by $v_1, v_3, v_4, v_5$ since $d(u)\ge 4$, and it follows that $d(v_2)=3$, a contradiction. So $G[N(v)]= C$ and then $d(u)=5$, $N(u)=N(v)$ and $G=WW_7$. By Lemma \ref{lem:N_K13+}, 
$$\mathcal{N}(K_{1,3}^+,G)=4\times 5+5\times10=70<4n^2-12n-4=108.$$ 

So we assume that $n\ge 8$. Let $xy\in E(G)$. Since $G$ is a planar graph, $G[N(y)-x]$ is $K_4$-minor-free. By Lemma \ref{lem:K4minor}, $e(G[N(y)-x])\le 2|N(y)-x|-3=2d(y)-5$. By Lemma \ref{lem:N_K13+}, we have
\begin{equation*}\label{equ:N_K13+}
 \begin{aligned}
\mathcal{N}(K_{1,3}^+,G) &\le \sum\limits_{xy\in E(G)}\left((2d(x)-5)+(2d(y)-5)\right)\\
 &=2\sum\limits_{xy\in E(G)}\left(d(x)+d(y)\right)-10(3n-6)\\
 &=2\sum\limits_{x\in V(G)}d^2(x)-30n+60. 
 \end{aligned}
\end{equation*}
It follows from Lemma \ref{lem:delta4_d2x} that 
$$\mathcal{N}(K_{1,3}^+,G)\le 2(2n^2+8n -24)-30n+60=4n^2-12n-4 -2(n-8)\le 4n^2-12n-4,$$ 
with equality only if $n=8$ and $G$ has two $6$-vertices and six $4$-vertices. 
We claim that $\mathcal{N}(K_{1,3}^+,G)<4n^2-12n-4$.  
Let $v\in V(G)$ with $d(v)=6$, $V(G)\setminus N[v] = \{u\}$ and $N(v)=\{v_1, \dots, v_6\}$ such that $G[N(v)]$ contains a cycle $v_1v_2\dots v_6v_1$. If $v_1v_3\in E(G)$, then $uv_2\in E(G)$( for otherwise $d(u_2)= 3$, a contradiction) and it follows that $N(u) =\{v_1, v_2, v_3\}$ and $d(u)=3$, a contradiction. So $v_1v_3\notin E(G)$. Similarly, $v_1v_5\notin E(G)$. Thus $d(v_1)\le 5$ and then $d(v_1) =4$. By the same arguments, we have $d(v_i)=4$ for any $i\in \{1, 2, \dots, 6\}$. So $d(u)=6$. It follows that $N(u)=N(v)$ and $G=WW_8$. By Lemma \ref{lem:N_K13+}, 
\begin{equation*}
 \mathcal{N}(K_{1,3}^+,G)=4\times 6+6\times12=96<4n^2-12n-4=156,
\end{equation*}
as claimed.
\end{proof}
Let $r,s,t\ge1$ be integers. We define four graphs as follows. 

\begin{itemize}
    \item $G_1 = H_5$ with 
    $V(G_1)=\{v, v_1, v_2, v_3, w\}$, where $N_{G_1}(v)=\{v_1,v_2,v_3\}$.
    
    \item $G_2 = H_{r+2}$ with 
    $V(G_2)=\{a_1,\dots,a_r,z_1,z_2\}$, where $a_1a_2\dots a_r=P_r$ and 
    $d_{G_2}(z_1)=d_{G_2}(z_2)=r+1$.
    
    \item $G_3 = H_{s+2}$ with  
    $V(G_3)=\{b_1,\dots,b_s,z_3,z_4\}$, where $b_1b_2\dots b_s=P_s$ and 
    $d_{G_3}(z_3)=d_{G_3}(z_4)=s+1$.
    
    \item $G_4 = H_{t+2}$ with 
    $V(G_4)=\{c_1,\dots,c_t,z_5,z_6\}$, where $c_1c_2\dots c_t=P_t$ and 
    $d_{G_4}(z_5)=d_{G_4}(z_6)=t+1$.
\end{itemize}

\begin{figure}
  \centering
  \includegraphics[width=7cm]{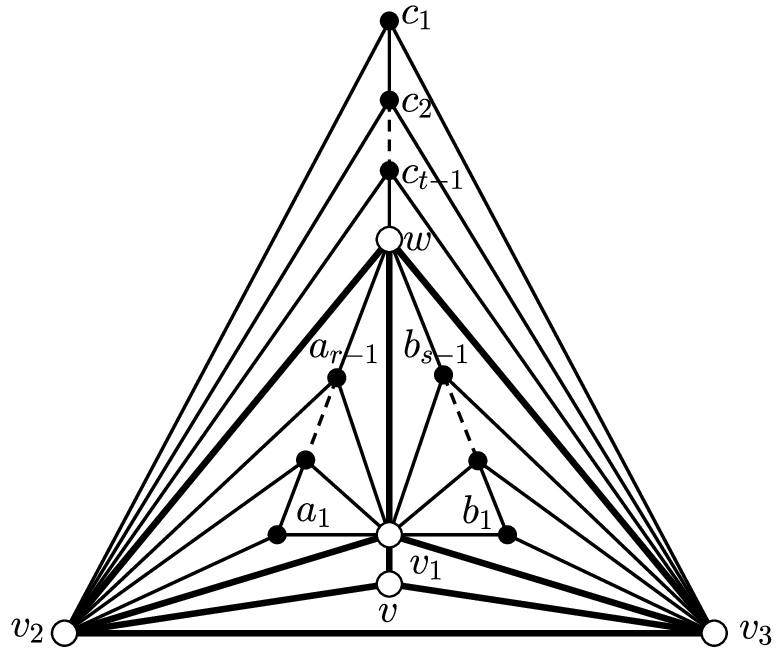}
  \caption{$G_{r,s,t}$.}\label{Fig:G_rst}
\end{figure}

Define $G_{r,s,t}$ as the graph obtained from the disjoint union of $G_1,G_2,G_3$ and $G_4$ by identifying the vertices in each of the sets  $\{w,a_r,b_s,c_t\}$, $\{v_1,z_1,z_3\}$, $\{v_2,z_2,z_5\}$, and $\{v_3,z_4,z_6\}$, respectively, into a single vertex (see Figure~\ref{Fig:G_rst}). 
Clearly, $G_{r,s,t}$ is a triangulation with $|G_{r,s,t}|=r+s+t+2$. 
Let $\mathcal{G}=\{G_{r,s,t}: r,s,t\ge1\}$.

\begin{lem}\label{lem:iff}
Let $G$ be a maximal planar graph. Then $G\in \mathcal{G}$ if and only if $G$ has a vertex $v$ of degree $3$ with $N(v)=\{v_1,v_2,v_3\}$ satisfying the following two conditions:
\begin{enumerate}[{\rm(i)}]
 \item \label{i} there exists exactly one vertex $w\in V(G)\setminus N[v]$ such that $|N(w)\cap N(v)|=3;$
 \item \label{ii} For any vertex $u\in V(G)\setminus(N[v]\cup\{w\})$, $|N(u) \cap N(v)|=2$.
\end{enumerate}
\end{lem}

\begin{proof}
The necessity follows immediately from the definition of $\mathcal{G}$, so we prove the sufficiency.
Clearly, $G[N[v]]=K_4$.
By~\eqref{i}, we have $G[N[v]\cup\{w\}]=H_5$.
\begin{figure}
  \centering
  \includegraphics[width=4cm]{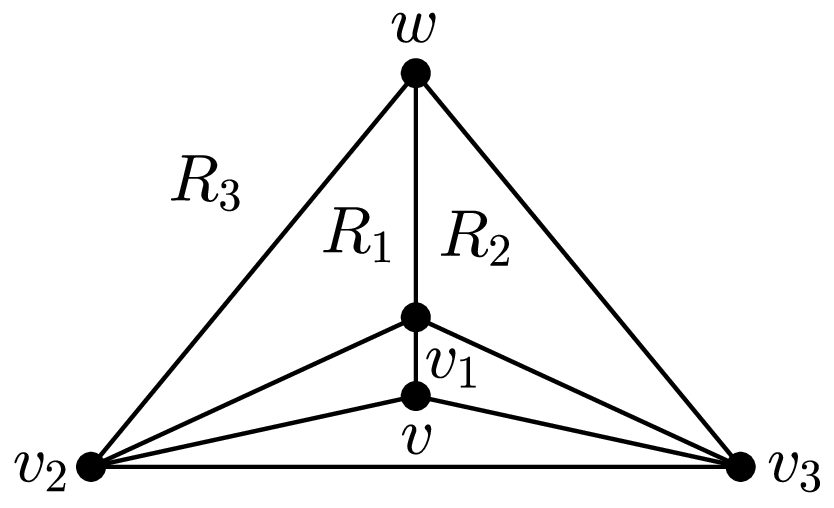}
  \caption{The local structure in a planar embedding of $G$.}\label{Fig:Ri}
\end{figure}
Embed $G$ in the plane so that $v$ lies inside the triangle $wv_2v_3$ (see Figure~\ref{Fig:Ri}), and let $R_1, R_2,$ and $R_3$ be the regions bounded by the triangles $wv_1v_2$, $wv_1v_3$, and $wv_2v_3$, respectively, none of which contains $v$.
Denote by ${\rm Int}(R_i)$ the set of vertices strictly inside $R_i$.

If ${\rm Int}(R_1)=\emptyset$, then $G[\{v_1,v_2,w\}\cup{\rm Int}(R_1)]=G[\{v_1,v_2,w\}]=H_3$.
Otherwise, by~\eqref{ii}, every vertex $a\in{\rm Int}(R_1)$ is adjacent to both $v_1$ and $v_2$, and hence ${\rm Int}(R_1)\subseteq N(v_1)$.
Let $N(v_1)\cap{\rm Int}(R_1)=\{a_1,\dots,a_{r-1}\}$ with $r\ge2$, listed in clockwise order, and set $a_r=w$.
Then, for $1\le i\le r-1$, we have $a_iv_1,a_iv_2\in E(G)$ and no vertex lies inside the cycle $v_1a_iv_2a_{i+1}v_1$, which implies $a_ia_{i+1}\in E(G)$.
Thus, $G[\{v_1,v_2,w\}\cup{\rm Int}(R_1)]=H_{r+2}$.

Similarly, for some integers $s,t\ge1$, we have $G[\{v_1,v_3,w\}\cup{\rm Int}(R_2)]=H_{s+2}$ and $G[\{v_2,v_3,w\}\cup{\rm Int}(R_3)]=H_{t+2}$.
Therefore, $G=G_{r,s,t}\in\mathcal{G}$.
\end{proof}

A vertex $v\in V(G)$ is called a \textit{central vertex} of $G$ if $d_G(v)=3$ and $v$ satisfies the two conditions {\rm(i)} and {\rm(ii)} of Lemma~\ref{lem:iff}. 
Let $C(G)$ be the set of all central vertices of $G$. For any graph $H$, let $V_3(H)=\{x\in V(H): d_H(x)=3\}$. By definition, $C(G)\subseteq V_3(G)$.

\begin{lem}\label{lem:center}
Let $G$ be a graph with $n\ge 6$ vertices. Then
$G\in\mathcal{G}$ and $G-x\in \mathcal{H}^+$ for some $x\in C(G)$ if and only if $G\in\mathcal{H}^+$.
\end{lem}

\begin{proof}
We first prove the sufficiency.
Let $G \in \mathcal{H}^+$. Observe that
$J_1 = G_{1,2,2}$, $J_2 = G_{2,2,2}$, and $H_n = G_{1,1,n-4}$ for $n \ge 6$.
Thus $\mathcal{H}^+ \subseteq \mathcal{G}$.
Moreover, for each $H \in \mathcal{H}^+$, we have $V_3(H)=C(H)$, and for every $v\in C(H)$, $H-v\in\mathcal{H}^+$.
Hence the sufficiency follows.

We now prove the necessity.
Assume that $G=G_{r,s,t}$ for some $1\le r\le s\le t$, and $G-x\in \mathcal{H}^+$ for some $x\in C(G)$.
We use the vertex labels as in Figure~\ref{Fig:G_rst}. Suppose that $G\notin\mathcal{H}^+$.
Then $t\ge 3$.
Since $x\in C(G)$, for each $y\in V(G)\setminus N[x]$ we have
$|N_G(x)\cap N_G(y)|\ge 2$.
If $r=1$, then $t\ge s\ge 2$ and $V_3(G)=\{v,b_1,c_1\}$;
if $r\ge 2$, then $V_3(G)=\{v,a_1,b_1,c_1\}$.
In both cases,
$|N_G(b_1)\cap N_G(c_1)|=1$, and when $r\ge 2$ we also have
$|N_G(a_1)\cap N_G(c_1)|=1$.
Thus $\left(V_3(G)\setminus\{v\}\right)\cap C(G)=\emptyset$.
Since $C(G)\subseteq V_3(G)$, it follows that $x=v$.
Consequently,
$V_3(G-x)=V_3(G)\setminus\{x\}\subseteq\{a_1,b_1,c_1\}$,
and for each $y\in V_3(G-x)\setminus\{c_1\}$ we have
$|N_{G-x}(y)\cap N_{G-x}(c_1)|=1$.
Hence $C(G-x)=\emptyset$.
However, since $\mathcal{H}^+\subseteq\mathcal{G}$, we must have
$C(G-x)\neq\emptyset$, a contradiction.
Therefore, $G\in\mathcal{H}^+$.
\end{proof}

\begin{lem}\label{lem:delta3}
Let $G$ be a maximal planar graph with $n\ge6$ vertices, and let $v$ be a $3$-vertex of $G$. Then
$$\mathcal{N}_v(K_{1,3}^+,G)\le 8n-16,$$ 
with equality if and only if $G\in\mathcal{G}$ and $v\in C(G)$.
\end{lem}
\begin{proof}
Let $N(v)=\{v_1,v_2,v_3\}$. Since $G$ is a maximal planar graph, $G[N(v)]$ forms a triangle $v_1v_2v_3v_1$. 
We partition all $K_{1,3}^+$-subgraphs of $G$ containing $v$ into three subsets $\mathcal{S}_1, \mathcal{S}_2$ and $\mathcal{S}_3$ such that for any $i\in \{1,2,3\}$, $d_H(v)=i$ for all $H\in \mathcal{S}_i$. Let $s_i=|\mathcal{S}_i|(i=1,2,3)$. Then $$\mathcal{N}_v(K_{1,3}^+,G)=s_1+s_2+s_3.$$

If $H\in\mathcal{S}_3$, then $V(H)=N[v]$. So $s_3=3$. Suppose that $H\in\mathcal{S}_2$ and $vv_1$, $vv_2\in E(H)$. Then $v_1v_2\in E(H)$. If $d_H(v_1)=3$, then there is a vertex $u\in N_G(v_1)\setminus \{v, v_2\}$ such that $\{uv_1\}=E(H)\setminus \{vv_1, vv_2, v_1v_2\}$. Thus, the number of subgraphs in $\mathcal{S}_2$ containing the edge $v_1v_2$ is $(d(v_1)-2)+(d(v_2)-2)$. Summing over all three edges $v_1v_2,v_2v_3$ and $v_3v_1$, we obtain $$s_2=(d(v_1)+d(v_2)-4)+(d(v_2)+d(v_3)-4)+(d(v_3)+d(v_1)-4)=2\sum\limits_{i=1}^3d(v_i)-12.$$

Now we begin to consider $\mathcal{S}_1$. Then for any $H\in\mathcal{S}_1$, $|V(H)\setminus N[v]| \le 2$. Thus we continually partition $\mathcal{S}_1$ into three subsets: $\mathcal{S}_1^i =\{H\in\mathcal{S}_1:\; |V(H)\setminus N[v]| =i\}(0\le i \le 2)$. Let $s_1^i=|\mathcal{S}_1^i|$ for any $i\in\{0,1,2\}$. Then
$$\mathcal{S}_1=\mathcal{S}_1^0\cup \mathcal{S}_1^1\cup \mathcal{S}_1^2 \text{ and } s_1=s_1^0 + s_1^1 + s_1^2.$$

Clearly, $s_1^0=3$. For any $H\in \mathcal{S}_1^1$, there is exactly one vertex in $V(G)\setminus N[v]$. Thus for any $w \in V(G)\setminus N[v]$, let $s_1^1(w)$ be the number of subgraphs of $G$ in $\mathcal{S}_1^1$ containing $w$. By the definition of $\mathcal{S}_1^1$, $s_1^1(w)=0$ if $|N(w)\cap N(v)|\le 1$. If $|N(w)\cap N(v)|=2$, then $s_1^1(w)=2$. If $|N(w)\cap \{v_1,v_2,v_3\}|=3$, then $s_1^1(w)=6$ (see Figure \ref{Fig:6K13}). By the planarity of $G$, at most one vertex can be adjacent to all three neighbors of $v$. Since $|V(G)\setminus N[v]|=n-4$, 
\begin{equation*}\label{equ:s11}
 s_1^1\le6+2(n-4-1)=2n-4,
\end{equation*}
with equality if and only if there is only one vertex $w\in V(G)\setminus N[v]$ that satisfies $|N(w)\cap N(v)|=3$ and $|N(x)\cap N(v)|=2$ for any $x\in V(G)\setminus (N[v]\cup \{w\})$. 
By Lemma \ref{lem:iff}, $s_1^1 = 2n-4$ if and only if $G\in \mathcal{G}$ and $v\in C(G)$. 

\begin{figure}
 \centering
 \includegraphics[width=17cm]{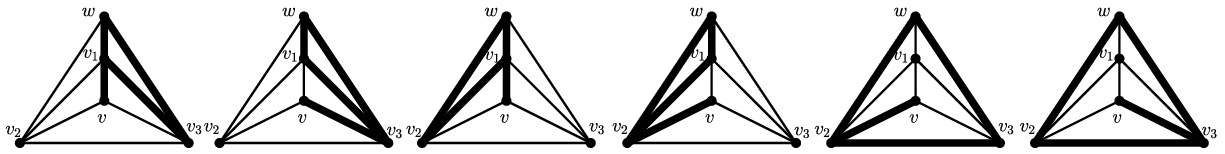}
 \caption{Six subgraphs containing $w$ with exactly one vertex outside $N[v]$.}\label{Fig:6K13}
\end{figure}

For any $H\in \mathcal{S}_1^2$, $|E(H)\cap E(G-N[v])| = 1$. For any $xy\in E(G-N[v])$, $|N(x)\cap N(y) \cap N(v)|\le 2$. Since $|E(G-N[v])| = 3n - 6 - (\sum_{i=1}^3d(v_i)-3)$, 
\begin{equation*}\label{equ:s12}
 s_1^2\le 2\left(3n-3-\sum\limits_{i=1}^3d(v_i)\right),
\end{equation*}
with equality if and only if the following condition holds:
\begin{gather*}
\tag{$*$}
|N(a)\cap N(b) \cap N(v)|=2 \quad \text{for all } ab\in E(G-N[v]).
\end{gather*}

Finally, we have
\begin{align*}
 \mathcal{N}_v(K_{1,3}^+,G) &=(s_1^0+s_1^1+s_1^2)+s_2+s_3 \\
 &\le\left(3+(2n-4)+2\left(3n-3-\sum\limits_{i=1}^3d(v_i)\right)\right)+\left(2\sum\limits_{i=1}^3d(v_i)-12\right)+3\\
 &=8n-16,
\end{align*}
with equality if and only if $G \in \mathcal{G}$ and $v \in C(G)$ ($(*)$ holds automatically).
Thus, Lemma \ref{lem:delta3} holds.
 \end{proof}
 
Now, we begin to prove Theorem {\rm\ref{thm:K13+}}.
\begin{proof}[Proof of Theorem {\rm\ref{thm:K13+}}]
Let $G$ be an $n (n\ge4)$-vertex planar graph and $\mathcal{N}(K_{1,3}^+, G) = f(n,K_{1,3}^+)$. 
If $n\in\{4,5\}$, it follows from Lemma \ref{lem:maximal} that $\mathbb{F}(n,K_{1,3}^+)=\{H_n\}$. 
By Lemma \ref{lem:N_K13+}, 
$$f(4,K_{1,3}^+)=\mathcal{N}(K_{1,3}^+, H_4)=12=4n^2-12n-4 \text{ and } f(5,K_{1,3}^+)=\mathcal{N}(K_{1,3}^+, H_5)=36=4n^2-12n-4.$$

Assume $n\ge6$. By Lemma \ref{lem:maximal}, $G$ is a triangulation and $\delta(G)\ge3$. 
By Lemma \ref{lem:K13+,delta4}, $\mathcal{N}(K_{1,3}^+,G) < 4n^2-12n-4$ when $\delta(G)\ge4$. 
Assume $\delta(G)=3$ and let $v$ be a 3-vertex of $G$. By induction and Lemma \ref{lem:delta3}, $$\mathcal{N}(K_{1,3}^+,G)=\mathcal{N}(K_{1,3}^+,G-v)+\mathcal{N}_v(K_{1,3}^+,G)\le 4(n-1)^2-12(n-1)-4+(8n-16)=4n^2-12n-4,$$ 
with equality if and only if $G-v\in \mathcal{H}^+$ and $G\in\mathcal{G}$ with $v\in C(G)$.  
Thus, $f(n,K_{1,3}^+)=4n^2-12n-4$ and $\mathbb{F}(K_{1,3}^+)=\mathcal{H}^+$ by Lemma \ref{lem:center}. 
\end{proof}

\section{Proof of Theorem \ref{thm:K4-}}

\begin{proof}[Proof of Theorem {\rm \ref{thm:K4-}}]
We proceed by induction on the number of vertices. Let $G$ be an $n$-vertex planar graph with $\mathcal{N}(K_4^-,G)=f(n,K_4^-)$.
By Lemma \ref{lem:maximal}, $G$ is a triangulation and $3\le \delta(G)\le5$. If $n\in\{4,5\}$, by Lemma \ref{lem:maximal}, $\mathbb{F}(n,K_4^-)=\{H_n\}$ and a direct count gives that 
$$f(n,K_4^-)=\mathcal{N}(K_4^-,H_n)=\frac{1}{2}\left(n^2+9n-40\right). $$

In the following, we assume $n\ge6$. Suppose that $\delta(G)\ge4$. Let $x\in V(G)$. If $G$ has a $K_4^-$-copy $H$ such that $d_H(x)=3$, then $H\cap G[N(x)] = P_3$. Thus we count the number, denoted by $\mathcal{N}_x(K_4^-,G,3)$, of $K_4^-$-copies containing $x$ such that $x$ has degree three in the subgraph. Note that $G[N(x)]$ is an outerplanar graph. By Lemma \ref{lem:exop_P3} and Lemma \ref{lem:delta4_d2x}, we have
\begin{align*}
 \mathcal{N}(K_4^-,G) & = \frac{1}{2}\sum\limits_{x\in V(G)}\mathcal{N}_x(K_4^-,G,3)\\
 & \le \frac{1}{2}\sum\limits_{x\in V(G)} \frac{d^2(x)+3d(x)-12}{2}\\
 & = \frac{1}{4}\left(\sum\limits_{x\in V(G)} d^2(x)+ 3\sum\limits_{x\in V(G)}d(x)-12n\right)\\
 & \le \frac{1}{4}\left((2n^2+8n-24)+3(6n-12)-12n\right)\\
 & = \frac{1}{2}\left(n^2+7n-30\right)<\frac{1}{2}\left(n^2+9n-40\right).
\end{align*}

Now assume that $\delta(G)=3$. Let $v$ be a $3$-vertex of $G$ and $N(v)=\{v_1,v_2,v_3\}$, where $v_1v_2,v_2v_3,v_3v_1\in E(G)$. 
We now partition the set $\mathcal{S}$ of $K_4^-$-subgraphs containing $v$ into two disjoint classes: $S_i =\{ H\in \mathcal{S}$: $|V(H)\setminus N[v]|= i\}$ for each $i\in \{0, 1\}$. Let $s_i=|\mathcal{S}_i|$ for $i\in\{0,1\}$. Since $G[N[v]] = K_4$, $s_0 = \mathcal{N}(K_4^-,G[N[v]]) = \mathcal{N}(K_4^-,K_4) = 6$. 
For each $H\in\mathcal{S}_1$, let $\{x, y\}= V(H)\cap N(v)$. Then $xy\in \{v_1v_2,v_2v_3,v_3v_1\}$. 
Define $A_i = N(v_i)\cap N(v_{i+1})\setminus N[v]$ for $i\in\{1,2,3\}$ where $v_4=v_1$. Since $G$ is planar, $|A_i\cap A_j|\le1$ for $i\neq j$. Thus, 
\begin{align*}\label{equ:K4-:s2}
s_1&= \sum\limits_{i=1}^{3} |(N(v_i)\cap N(v_{i+1}))\setminus N[v]|= \sum\limits_{i=1}^{3}|A_i|\le|A_1\cup A_2\cup A_3|+2\le |V(G)\setminus N[v]|+2=n-2,
\end{align*}
with equality if and only if $|N(w)\cap N(v)|=3$ for some vertex $w\in V(G)\setminus N[v]$ and $|N(w')\cap N(v)|=2$ for each $w'\in V(G)\setminus N[v]-w$. By Lemma \ref{lem:iff}, $s_1=n-2$ if and only if $G\in\mathcal{G}$ and $v\in C(G)$. 
So 
$$\mathcal{N}_v(K_4^-,G)=s_0+s_1 \le 6+(n-2) = n+4.$$ 
By the induction hypothesis, we have
$$\mathcal{N}(K_4^-,G)=\mathcal{N}(K_4^-,G-v)+\mathcal{N}_v(K_4^-,G)\le \frac{1}{2}\left((n-1)^2+9(n-1)-40\right) + (n+4)= \frac{1}{2}\left(n^2+9n-40\right),$$
with equality if and only if $G-v\in\mathcal{H}^+$, $G\in\mathcal{G}$ and $v\in C(G)$. 
Thus, $f(n,K_4^-)=\frac{1}{2}\left(n^2+9n-40\right)$ and $\mathbb{F}(K_4^-)=\mathcal{H}^+$ by Lemma \ref{lem:center}. 
\end{proof}

\section{\texorpdfstring{The graphs in $\mathbb{F}(P_3)$ and $\mathbb{F}(K_4)$}{The graphs in F(P3) and F(K4)}}
\noindent

\begin{lem}[{\rm\cite{Alon1984Caro}}]\label{lem:d1d2d3}
 If $d_1,d_2$ and $d_3$ are the degrees of three different vertices of a planar graph with $n$ vertices, then $$d_1+d_2+d_3\le 2n+2.$$
\end{lem}
\begin{thm}\label{thm:P3}
$\mathbb{F}(P_3)=\mathcal{H}^+\cup\{H_3\}$.
\end{thm}

\begin{proof}
Let $G$ be an $n (n\ge3)$-vertex planar graph and $\mathcal{N}(P_3,G) = f(n,P_3)$. 
By Lemma \ref{lem:maximal}, $G$ is a triangulation.
For $3\le n\le 5$, it is clear that $G=H_n$. 
Now consider $n\ge6$. 
Then $$\mathcal{N}(P_3,G) = \mathcal{N}(K_{1,2},G)=\sum\limits_{x\in V(G)}\binom{d(x)}{2} = \frac{1}{2}\sum\limits_{x\in V(G)}\left(d^2(x)-d(x)\right) = \frac{1}{2}\sum\limits_{x\in V(G)}^nd^2(x)-e(G).$$
If $\delta(G)\ge4$, by Lemma \ref{lem:delta4_d2x}, $\mathcal{N}(P_3,G)\le \frac{1}{2}(2n^2+8n-24) - (3n-6) = n^2+n-6<f(n,P_3)=n^2+3n-16$. Now, assume $\delta(G)=3$. Let $v\in V(G)$ with $d_G(v)=3$ and $N_G(v)=\{u_1,u_2,u_3\}$. Denote $G'=G-v$. Then $$\mathcal{N}(P_3,G)=\mathcal{N}(P_3,G')+\mathcal{N}_v(P_3,G).$$ By induction, $\mathcal{N}(P_3,G')\le f(n-1,P_3)=(n-1)^2+3(n-1)-16=n^2+n-18$. Moreover, by Lemma \ref{lem:d1d2d3}, $$\mathcal{N}_v(P_3,G) = \binom{3}{2}+\sum_{i=1}^3\binom{d_G(u_i)-1}{1}=3+\left(\sum_{i=1}^3d_G(u_i)-3\right)\le3+(2n+2)-3=2n+2.$$ So $$\mathcal{N}(P_3,G)\le n^2+n-18+(2n+2)=n^2+3n-16=f(n,P_3),$$ with equality if and only if $G'\in\mathbb{F}(P_3)$ and $\sum_{i=1}^3d_G(u_i)=2n+2$, i.e. $\sum_{i=1}^3d_{G'}(u_i)=2n-1$. 

If $n=6$, then $G'=H_5$ and $\sum_{i=1}^3d_{G'}(u_i)=11$. Since $d_{G'}(u_i)\le n-2=4$ for $i\in\{1,2,3\}$, the degrees of vertices $u_1,u_2,u_3$ in $G'$ must be $3,4,4$. So vertex $v$ can be placed into any face of $f_i(1\le i \le 6)$ (see Figure \ref{Fig:Fn}$(a)$), yielding a unique non-isomorphic graph $H_6$.

If $n=7$, then $G'=H_6$ and $\sum_{i=1}^3d_{G'}(u_i)=13$. Since $d_{G'}(u_i)\le n-2=5$ for $i\in\{1,2,3\}$, the degrees of vertices $u_1,u_2,u_3$ in $G'$ are $3,5,5$ or $4,4,5$. So vertex $v$ can be placed into $f_1,f_2,f_5$ and $f_6$ (see Figure \ref{Fig:Fn}$(a)$), yielding two non-isomorphic graphs $H_7$ and $J_1$ (see Figure \ref{Fig:Fn}$(b)$). Thus, $\mathbb{F}(7,P_3)=\{H_7,J_1\}$. 

If $n=8$, then $G'\in\{H_7,J_1\}$ and $\sum_{i=1}^3d_{G'}(u_i)=15$. If $G'=H_7$, then placing $v$ into $f_1$ or $f_2$ yields $G=H_8$; otherwise $\sum_{i=1}^3d_{G'}(u_i)<15$. If $G'=J_1$, then placing $v$ into $f_6$ (see Figure \ref{Fig:Fn} $(b)$) yields $G=J_2$; otherwise the sum is again less than 15. So $\mathbb{F}(8,P_3)=\{H_8,J_2\}$. 

If $n=9$, then $G'\in\{H_8,J_2\}$ and $\sum_{i=1}^3d_{G'}(u_i)=17$. If $G'=H_8$, then placing $v$ into $f_1$ or $f_2$ yields $G=H_9$. If $G'=J_2$, then there is no face $f$ such that $\sum_{u\in\partial f}d(u)=17$. So $\mathbb{F}(9,P_3)=\{H_9\}$. 

Suppose $n=10$, then $G'=H_9$ and $\sum_{i=1}^3d_{G'}(u_i)=2n-1=3+2(n-2)$. Observe that the degree sequence of $G'$ is $3,3,4,\dots,4,n-2,n-2$. If $v$ is adjacent to at most one vertex with degree $n-2$, then $\sum_{i=1}^3d_{G'}(u_i)\le n-2+4\times2=n+6<2n-1$. So vertex $v$ is adjacent to two $(n-2)$-vertices and one 3-vertex, that is, $v$ is placed into face $f_1$ or $f_2$ (see Figure \ref{Fig:Fn}$(a)$), yielding a unique non-isomorphic graph $G=H_n$. Thus, $\mathbb{F}(n,P_3)=\{H_n\}$ for $n\ge10$. 
\end{proof}
\begin{lem}[\cite{Wood2007}]\label{lem:s_tri}
Let $G$ be an $n (n\ge5)$-vertex planar graph. Then $\mathcal{N}(K_4,G)=f(n,K_4)$ if and only if $G$ has a separating triangle $T$ and $\mathcal{N}(K_4,G_i)=f(n_i,K_4)$ for $i\in\{1,2\}$, where $G_1$ and $G_2$ are two induced subgraphs of $G$ such that $G_1\cup G_2=G$ and $G_1\cap G_2=T$. 
\end{lem}
\begin{thm}\label{thm:K4}
$\mathbb{F}(K_4)=\mathcal{A}\setminus\{K_3\}$. 
\end{thm}
\begin{proof}
Let $G$ be an $n$-vertex planar graph with $\mathcal{N}(K_4,G)=f(n,K_4)=n-3$. We are going to show that $G\in\mathcal{A}_n$. 
If $n=4$, then $G=K_4\in\mathcal{A}_4$. Assume $n\ge5$ and $T,G_1,G_2$ be defined as in Lemma \ref{lem:s_tri}. We can assume $G_2-T$ is connected and take $T$ such that $|G_2|$ is as small as possible. Then $\mathcal{N}(K_4,G_i)=f(|G_i|,K_4)$ for $i\in\{1,2\}$. By the definition of the Apollonian graphs and the choice of $G_2$, we have $|G_2-T|=1$. By induction, $G_i\in\mathcal{A}_{|G_i|}$. We claim that $G_1-T$ is connected. Suppose for a contradiction that $C_1$ and $C_2$ are two components of $G_1-T$. Then $N(C_1),N(C_2)\subseteq V(T)$. If $N(C_1)=N(C_2)=V(T)$, then $G$ contains a $K_{3,3}$ minor by contracting $C_1,C_2$ and $G_2-T$ into single vertices, respectively. We may assume $|N(C_1)|\le 2$. Let $G_1' = G[N[C_1]]$ and $G_2' = G-C_1$. Then $G_1'\cup G_2'=G$ and $|G_1'\cap G_2'|\le 2$. It follows that 
\begin{align*}
 n-3&=\mathcal{N}(K_4,G)= \mathcal{N}(K_4,G_1')+\mathcal{N}(K_4,G_2')\le f(|G_1'|,K_4)+f(|G_2'|,K_4)\\
 &= (|G_1'|-3)+(|G_2'|-3)=n+|G_1'\cap G_2'|-6\le n-4,
\end{align*}
a contradiction. This proves our claim, implying that $T$ can not be a separating triangle in $G_1$. So $G_1$ admits a planar embedding with all vertices of $G_1-T$ lie in the exterior of $T$. Since $|G_2-T|=1$, the unique vertex in $G_2-T$ can be embedded into the interior of $T$. Therefore, $G\in\mathcal{A}_n$. 
\end{proof}

\end{document}